\documentclass[english,microtype]{jloganal}
\usepackage{txfonts}
\usepackage{babel}
\usepackage{verbatim}
\usepackage[utf8]{inputenc}
\usepackage{qtree} 
\usepackage{pinlabel}
\usepackage{float}
\usepackage{pinlabel}

\newcommand{\spec}{\mathrm{Spec}}
\newcommand{\ord}{\mathrm{ord}}

\newcommand{\h}{\;\#\;}

\newcommand{\posint}{\mathbb{Z^+}}

\newcommand{\ideal}[1]{\langle #1\rangle}
\newcommand{\servar}[1]{\langle\langle #1 \rangle\rangle}

\newtheorem{thm}{Theorem}[section]
\newtheorem{lem}[thm]{Lemma}
\newtheorem{cor}[thm]{Corollary}
\theoremstyle{definition}
\newtheorem{defn}[thm]{Definition}

\newtheorem{exmp}{Example}[section]

\title[Dynamic Newton--Puiseux Theorem] {Dynamic Newton--Puiseux Theorem}
\author[B Mannaa]{Bassel Mannaa}
\givenname{Bassel}
\surname{Mannaa}
\address{Department of Computer Science and Engineering\\
University of Gothenburg\\\newline
SE-412 96 \\ Gothenburg \\ Sweden}
\email{bassel.mannaa@cse.gu.se}
\urladdr{http://www.cse.chalmers.se/~bassel}
\author[T Coquand]{Thierry Coquand}
\givenname{Thierry}
\surname{Coquand}
\email{thierry.coquand@cse.gu.se}
\address{Department of Computer Science and Engineering\\
University of Gothenburg\\\newline
SE-412 96 \\ Gothenburg \\ Sweden}
\urladdr{http://www.cse.chalmers.se/~coquand}

\subject{primary}{MSC2010}{03F65, 14Q05,68W30}
\subject{secondary}{MSC2010}{12Y05, 12E05, 12F05}

\volumenumber{}
\issuenumber{}
\publicationyear{}
\papernumber{}
\startpage{}
\endpage{}
\doi{}
\MR{}
\Zbl{}
\received{}
\revised{}
\accepted{}
\published{}
\publishedonline{}
\proposed{}
\seconded{}
\corresponding{}
\editor{}
\version{}
\begin{document}

\begin{abstract}
A constructive version of Newton--Puiseux theorem for computing the Puiseux expansions of algebraic curves is presented. The proof is based on a classical proof by Abhyankar. Algebraic numbers are evaluated dynamically; hence the base field need not be algebraically closed and a factorization algorithm of polynomials over the base field is not needed. The extensions obtained are a type of regular algebras over the base field and the expansions are given as formal power series over these algebras.
\end{abstract}

\maketitle
\section*{Introduction}
Newton--Puiseux Theorem states that, for an algebraically closed field $K$ of zero characteristic, given a polynomial $F \in K[[X]][Y]$ there exist a positive integer $m$ and a factorization $F = \prod_{i=1}^n (Y - \eta_i)$ where each $\eta_i \in K[[X^{1/m}]][Y]$. These roots $\eta_i$ are called the \emph{Puiseux expansions} of $F$. The theorem was first proved by Newton \cite{newton} with the use of Newton polygon. Later, Puiseux \cite{Puiseux_recherchessur} gave an analytic proof. 
It is usually stated as: 
\emph{The field of fractional power series \footnote{Also known as the field of Puiseux series.}, i.e. the field $K\servar{X}=\bigcup_{m\in\posint} K((X^{1/m}))$, is algebraically closed} \cite{walker}. Abhyankar
\cite{Abhyankar90algebraicgeometry} presents another proof of this result, the
``Shreedharacharya's Proof of Newton's Theorem''. This proof is not constructive as it stands. Indeed it assumes decidable equality on the ring $K[[X]]$ of power series over a field, but given two arbitrary power series we cannot decide whether they are equal in \emph{finite} number of steps. We explain in this paper how to modify
his argument by adding a separability assumption to provide a constructive proof of the result:
The field of fractional power series is \emph{separably} algebraically closed.
In particular, the termination of Newton--Puiseux algorithm is justified constructively in this case.
This termination is justified by a non constructive reasoning in most references \cite{walker, duval2,Abhyankar90algebraicgeometry}, with 
the exception of \cite{edwards2005essays} (For an introduction to constructive algebra, see \cite{Mines,lombardi_book}). 
Following that, we show that the field of fractional power series algebraic over $K(X)$ is algebraically closed.

 Another contribution of this paper is to analyze in a constructive framework what happens if the field $K$ is 
not supposed to be algebraically closed.
The difference with \cite{edwards2005essays}, which provides also such an analysis, is that we do not assume the 
irreducibility of polynomials to be decidable.
This is achieved through the method of \emph{dynamic evaluation} \cite{duval0}, which replaces factorization by gcd computations. The reference  
\cite{Coste2001203} provides a proof theoretic analysis of this method. \\
With dynamic evaluation we obtain algebras, \emph{triangular separable algebras}, as separable extensions of the base field and the Puiseux expansions are given over these algebras. Theorem \ref{dynamicnewton-extended} shows that the extensions produced by the algorithm are minimal in the sense that if $R$ is one such extension and $A$ is any other algebra over the base field such that $F(X,Y)$ factors linearly over $A[[X^{1/r}]]$ for some positive integer $r$, then $A$ splits $R$ which in case $A$ and $R$ were fields would be equivalent to saying that $A$ contains the normal closure of $R$. But this then shows that $R$ splits itself, which in case $R$ is a field is equivalent to saying that $R$ is a normal extension (Corollary \ref{spliteachother}). Theorem \ref{dynamicnormal2} will then show that any two triangular separable algebras $A$ and $B$ that split each other are in fact powers of a some triangular separable algebra, i.e. $A \cong R^m$ and $B \cong R^n$ for some triangular separable algebra $R$ and positive integers $m,n$.

This algorithm gives less information than Duval's rational Puiseux expansion algorithm \cite{duval2} since we can easily obtain the classical Puiseux  expansions of a polynomial from the rational ones (Rational Puiseux expansions describe the roots of the polynomial by pairs of power series, i.e. a parametrization, with rational coefficients). In \cite{duval2} the rational expansions of a polynomial $F(X,Y) \in K[X,Y]$ are given as long as $F(X,Y)$ is absolutely irreducible, i.e. irreducible in $\bar{K}[X,Y]$, where $\bar{K}$ is the algebraic closure of $K$. It would be interesting to also justify Duval's algorithm in a constructive framework. 

\section{A constructive version of Abhyankar's Proof}

 We recall that a (discrete) field is defined to be a non trivial ring in which any element is $0$ or invertible. For a ring $R$, the formal power series ring $R[[X]]$ is the set of sequences $\alpha=\alpha(0) + \alpha(1) X + \alpha(2) X^2+...$, with $\alpha(i) \in R$ \cite{Mines}. 

An \emph{apartness} relation $\#$ on a set is a symmetric relation satisfying $x \h y \rightarrow x \h z \lor y \h z$ and $\neg x \h x$. An apartness is tight if it satisfies $\neg x \h y \rightarrow x = y$. In addition to the ring identities, a ring with apartness satisfies $x_1 + y_1 \h x_2 + y_2 \rightarrow x_1 \h x_2 \lor y_1 \h y_2$, $x_1 y_1 \h x_2 y_2 \rightarrow x_1 \h x_2 \lor y_1 \h y_2$ and $0\h1$, see \cite{Mines, vandalen}.

Next we define the apartness relation on power series as in \cite[Ch 8] {vandalen}.
\begin{defn}
\label{apartd}
Let $R$ be a ring with apartness. For $\alpha, \beta \in R[[X]]$ we define $\alpha \h \beta$ if $\exists n\; \alpha(n) \h \beta(n)$.  
\end{defn}

The relation $\#$ as defined above is an apartness relation and makes $R[[X]]$ into a ring with apartness \cite{vandalen}. This definition of $\#$ applies to the ring of polynomials $R[X] \subset R[[X]]$.

We note that, if $K$ is a discrete field then for $\alpha \in K[[X]]$ we have $\alpha \h 0$ iff $\alpha(j)$ is invertible for some $j$. For $F = \alpha_0 Y^n + ...+ \alpha_n \in K[[X]][Y]$, we have $F \h 0$ iff $\alpha_i(j)$ is invertible for some $j$ and $0\leq i\leq n$.

 Let $R$ be a commutative ring with apartness. Then $R$ is an \emph{integral domain} if it satisfies $x \h 0 \land y \h 0 \rightarrow  xy \h 0$ for all $x,y \in R$. A \emph{Heyting} field is an integral domain satisfying $x \h 0 \rightarrow \exists y\; xy = 1$. The Heyting field of fractions of $R$ is the Heyting field obtained by inverting the elements $c \h 0$ in $R$ and taking the quotient by the appropriate equivalence relation, see \cite [Ch 8,Theorem 3.12] {vandalen}. For $a$ and $b \h 0$ in R we have $a/b \h 0$ iff $a \h 0$.
  
For a discrete field $K$, an element $\alpha \h 0$ in $K[[X]]$ can be written as $X^m \sum_{i \in \mathbb{N}} a_i X^i$ with $m \in \mathbb{N}$ and $a_0 \neq 0$. It follows that the ring $K[[X]]$ is an integral domain. If $a_0 \neq 0$ we have that $\sum_{i \in \mathbb{N}} a_i X^i$ is invertible in $K[[X]]$. We denote by $K((X))$, the Heyting field of fractions of $K[[X]]$, we also call it the Heyting field of Laurent series over $K$. Thus an element apart from $0$ in $K((X))$ can be written as $X^n \sum_{i \in \mathbb{N}} a_i X^i$ with $a_0 \neq 0$ and $n \in \mathbb{Z}$, i.e. as a series where finitely many terms have negative exponents. 
 
 Unless otherwise qualified, in what follows, a field will always denote a discrete field.
  
\begin{defn}[Separable polynomial]
Let $R$ be a ring. A polynomial $p \in R[X]$ is separable if there exist $r,s\in R[X]$ such that $r p + s p' = 1$, where $p' \in R[X]$ is the derivative of $p$.
\end{defn}
\begin{lem}
\label{sepdivisorsep}
Let $R$ be a ring and $p \in R[X]$ separable. If $p = f g$ then both $f$ and $g$ are separable.
\end{lem}
\begin{proof}
Let $r p + s p' = 1$ for $r,s \in R[X]$. Then $r f g + s (fg' + f' g) = (rf + s f') g + sf g' = 1$, thus $g$ is separable. Similarly for $f$.
\end{proof}

\begin{lem}
\label{separablevarchangepersist}
Let $R$ be a ring. If $p(X) \in R[X]$ is separable and $u \in R$ a unit then $p(u Y) \in R[Y]$ is separable.
\end{lem}

The following result is usually proved with the assumption of existence of a decomposition into irreducible
factors. We give a proof without this assumption. It works over a field of any characteristic.

\begin{lem}
\label{squarefree}
Let $f$ be a monic polynomial in $K[X]$ where $K$ is a field. If $f'$ is the derivative of $f$ 
and $g$ monic is the gcd of $f$ and $f'$ then writing $f = hg$ we have that $h$ is separable.
We call $h$ the separable associate of $f$.
\end{lem}
\begin{proof} Let $a$ be the gcd of $h$ and $h'$. We have $h = l_1 a$.
Let $d$ be the gcd of $a$ and $a'$. We have $a = l_2 d$ and $a' = m_2 d$, with $l_2$ and
$m_2$ coprime.

 The polynomial $a$ divides $h' = l_1a' + l_1'a$ and hence that $a = l_2 d$ divides
$l_1 a' = l_1 m_2 d$. It follows that $l_2$ divides $l_1m_2$ and since $l_2$ and $m_2$
are coprime, that $l_2$ divides $l_1$.

 Also, if $a^n$ divides $p$ then $p = q a^n$ and $p' = q'a^n + nqa'a^{n-1}$. Hence
$d a^{n-1}$ divides $p'$. Since $l_2$ divides $l_1$, this implies 
that $a^n = l_2 d a^{n-1}$ divides $l_1 p'$. So $a^{n+1}$ divides $al_1 p' = h p'$.

 Since $a$ divides $f$ and $f'$, $a$ divides $g$. We show that $a^n$ divides $g$ for all $n$ by induction on $n$. If $a^n$ divides $g$
we have just seen that $a^{n+1}$ divides $g'h$. Also $a^{n+1}$ divides $h'g$ since $a$ divides $h'$.
So $a^{n+1}$ divides $g'h+h'g = f'$. On the other hand, $a^{n+1}$ divides $f = hg = l_1ag$.
So $a^{n+1}$ divides $g$ which is the gcd of $f$ and $f'$. 

 This implies that $a$ is a unit.
\end{proof}

 If $F$ is in $R[[X]][Y]$ we let $F_Y$ be the derivative of $F$ with respect to $Y$. 

\begin{lem} 
\label{apartness}
Let $K$ be a field and let $F  = \sum_{i = 0}^n \alpha_i Y^{n-i}\in K[[X]][Y]$ be separable over $K((X))$, then $\alpha_n \h 0 \lor \alpha_{n-1} \h 0$
\end{lem}
\begin{proof} 
Since $F$ is separable over $K((X))$ we have $P F + Q F_Y = \gamma \h 0$ for $P, Q \in K[[X]][Y]$ and $\gamma \in K[[X]]$. From this we get that $\gamma$ is equal to the constant term on the left hand side, i.e. $P(0) \alpha_n + Q(0) \alpha_{n-1} = \gamma \h 0$. Thus $\alpha_n \h 0 \lor \alpha_{n-1} \h 0$.
\end{proof}



 One key of Abhyankar's proof is Hensel's Lemma. We formulate a little more general version than the one
in \cite{Abhyankar90algebraicgeometry} by dropping the assumption that the base ring is a field.

\begin{lem}[Hensel's Lemma]
\label{hensel}
Let $R$ be a ring and $F(X,Y) = Y^n + \sum_{i=1}^n a_i(X)\;Y^{n-i}$ be a monic polynomial 
in $R[[X]][Y]$ of degree $n > 1$. 
Given monic $G_0, H_0\in R[Y]$ of degrees $r,s > 0$ respectively and $H^*, G^* \in R[Y]$ such that $F(0,Y) = G_0 H_0$,  $r + s = n$ and $G_0 H^* + H_0 G^*=1$; 
we can find $G(X,Y), H(X,Y) \in R[[X]][Y]$ of degrees $r, s$ respectively, such that $F(X,Y) = G(X,Y) H(X,Y)$ and $G(0,Y) = G_0$, $H(0,Y) = H_0$.
\end{lem}
\begin{proof}
The proof is almost the same as Abhyankar's \cite{Abhyankar90algebraicgeometry}, we present it here for completeness. \\
Since $R[[X]][Y] \subsetneq R[Y][[X]]$, we can rewrite $F(X,Y)$ as a power series in $X$ with coefficients in $R[Y]$. Let $F(X,Y) = F_0(Y) + F_1(Y) X + ...+ F_q(Y) X^q + ...$, with $F_i(Y) \in R[Y]$. Now we want to find $G(X,Y), H(X,Y) \in R[Y][[X]]$ such that $F = G H$. If we let $G = G_0+ \sum_{i=1}^{\infty} G_i(Y) X^i$ and $H = H_0 + \sum_{i=1}^{\infty} H_i(Y) X^i$, then for each $q$ we need to find $G_i(Y), H_j(Y)$ for $i,j \leq q$ such that $F_q = \sum_{i+j = q} G_i H_j$. We also need $\deg G_{k} < r$ and $\deg G_{\ell} < s$ for $ k, \ell > 0$.\\
We find such $G_i, H_j$ by induction on $q$. We have that $F_0 = G_0 H_0$. Assume that for some $q > 0$ we have found all $G_i, H_j$ with $\deg G_i < r$ and $\deg H_i < s$ for $1 \leq i < q$ and $1 \leq j < q$. Now we need to find $H_q, G_q$ such that\\
$F_q = G_0 H_q + H_0 G_q + \displaystyle\sum_{\substack{i+j = q \\ i < q, j < q}} G_i H_j$.
We let $U_q = F_q -  \displaystyle\sum_{\substack{i+j = q \\ i < q, j < q}} G_i H_j$, and we can see that $\deg U_q < n$. We are given that $G_0 H^* + H_0 G^* = 1$. Multiplying by $U_q$ we get $G_0 H^* U_q + H_0 G^* U_q = U_q$. By Euclidean division we can write $U_q H^* = E_q H_0 + H_q$ for some $E_q, H_q$ with $\deg H_q < s$. Thus we write $U_q = G_0 H_q + H_0 (E_q G_0 + G^* U_q)$. We can see that $\deg  H_0 (E_q G_0 + G^* U_q) < n$ since $\deg (U_q - G_0 H_q) < n$. Since $H_0$ is monic of degree $s$ , $\deg (E_q G_0 + G^* U_q) < r$. We take $G_q = E_q G_0 + G^* U_q$.\\
Now, we can write $G(X,Y), H(X,Y)$ as monic polynomials in $Y$ withe coefficients in $R[[X]]$, with degrees $r, s$ respectively.
\end{proof}

 It should be noted that the uniqueness of the factors $G$ and $H$ proven in \cite{Abhyankar90algebraicgeometry} may not necessarily hold when $R$ is not an integral domain.
 
 If $\alpha = \Sigma \alpha(i) X^i$ is an element of $R[[X]]$ we write $m\leqslant \ord~\alpha$ to mean that $\alpha(i)=0$
for $i<m$ and $m = \ord~\alpha$ to mean furthermore that $\alpha(m)$ is invertible. 

\begin{lem}\label{keylemma}
Let $K$ be an algebraically closed field of characteristic zero. \\
Let $F(X,Y) = Y^n + \sum_{i=1}^n \alpha_i(X) Y^{n-i}\in K[[X]][Y]$ be a
monic non-constant polynomial of degree $n \geq 2$ separable over $K((X))$. 
Then there exist $m>0$ and a proper factorization $F(T^m,Y)= G(T,Y)H(T,Y)$ with $G$ and $H$ in $K[[T]][Y].$
\end{lem}

\begin{proof}
We can assume w.l.o.g. that $\alpha_1(X) = 0$. This is Shreedharacharya's\footnote{Shreedharacharya's trick is also known as Tschirnhaus's trick \cite{vonTschirnhaus:2003:MRA}. The technique of removing the second term of a polynomial equation was also known to Descartes \cite{descartes_geom}.} trick \cite{Abhyankar90algebraicgeometry} (a simple change of variable $F(X, W-\alpha_1/n)$).
The simple case is if we have $\ord~\alpha_i = 0$ for some $1<i\leq n$. In this case $F(0,Y) = Y^n + d_2 Y^{n-1} + ... + d_n\in K[Y]$ and $d_i \neq 0$. Thus $\forall a \in K\; F(0,Y) \neq (Y-a)^n$. For any root $b$ of $F(0,b) = 0$
we have then a proper decomposition $F(0,Y) = (Y-b)^pH$ with $Y-b$ and $H$ coprime, and we can use Hensel's Lemma \ref{hensel} to conclude
(In this case we can take $m=1$).

In general, we know by Lemma \ref{apartness} that for $k=n$ or $k= n-1$ we have $\alpha_{k}(X)$ is apart from $0$. 
We then have $\alpha_{k}(\ell)$ invertible for some $\ell$. We can then find $p$ and $m$, $1 < m \leq n$,
such that $\alpha_m(p)$ is invertible and
$\alpha_i(j) = 0$ whenever $j/i  < p/m$ (See explanation below). We can then write
\begin{equation*}
F(T^m,T^{p}Z)= T^{np}(Z^n + c_2(T) Z^{n-2} + \dots + c_n(T))
\end{equation*}
with $\ord~c_m = 0$. As in the simple case, we have a proper decomposition 
$Z^n + c_2(T) Z^{n-2} + \dots + c_n(T) = G_1(T,Z)H_1(T,Z)$ with $G_1(T,Z)$ monic of degree $l$ in $Z$ and
$H_1(T,Z)$ monic of degree $q$ in $Z$, with $l+q = n,~l<n,~q<n$. 
We then take $G(T,Y) = T^{lp}G_1(T,Y/T^p)$ and $H(T,Y) = T^{qp}H_1(T,Y/T^p).$
\end{proof}
We note that since the polynomial is of finite $Y$ degree the search for $m$ and $p$ is finite. For example if the polynomial is of $Y$ degree $7$ (see Figure \ref{searchfinite}) and if $k = 4$ and $\ell=3$ we need only search the finite number of pairs to the left of the dotted line.
\begin{figure}[h]
\label{searchfinite}
\centering
\includegraphics[width=1.0\textwidth]{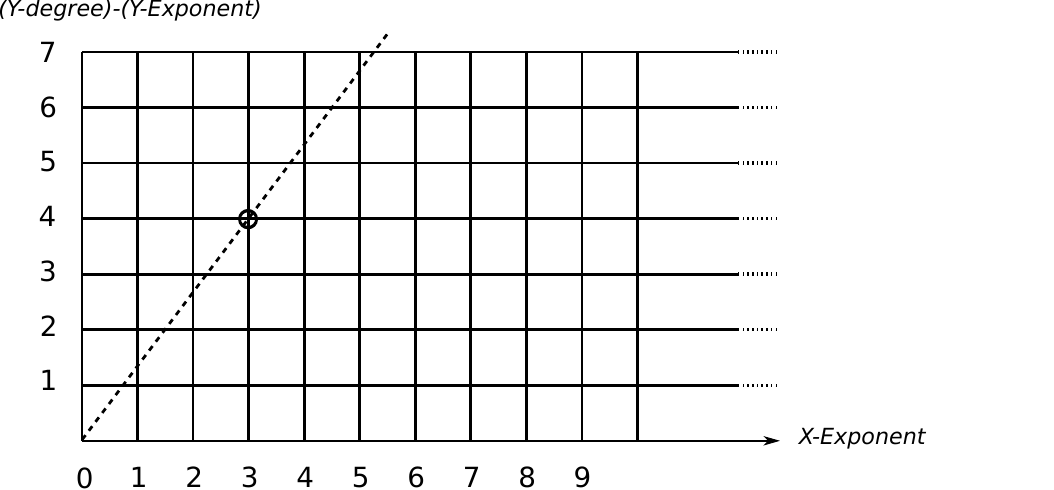}
\caption{Search for $m$ and $p$, Lemma \ref{keylemma}.}
\end{figure}
 
\begin{thm}\label{abhyankarproof}
Let $K$ be an algebraically closed field of characteristic zero.\\ 
Let $F(X,Y) = Y^n + \sum_{i=1}^n \alpha_i(X) Y^{n-i}\in K[[X]][Y]$ be a
monic non-constant polynomial separable over $K((X))$. Then there exist a positive integer $m$ and factorization
\begin{equation*}
F(T^m,Y)= \prod_{i=1}^{n} \big(Y - \eta_i\big)\;\;\;\;\;\;\;\; \eta_i \in K[[T]]
\end{equation*}
\end{thm}

\begin{proof}
If $F(X,Y)$ is separable over $K((X))$ then $F(T^m,Y)$ for some positive integer $m$ is separable over $K((T))$. The proof follows from Lemma \ref{sepdivisorsep} and Lemma \ref{keylemma} by induction.
\end{proof}

\begin{cor}
\label{sepalgclose}
Let $K$ be an algebraically closed field of characteristic zero. The Heyting field of fractional power series over $K$ is separably algebraically closed.
\end{cor}
\begin{proof}
Let $F (X,Y)\in K((X))[Y]$ be a monic separable polynomial of degree $n > 1$. Let $\beta \h 0$ be the product of the denominators of the coefficients of $F$. Then we can write $F(X, \beta^{-1}Z) = \beta^{-n} G$ for $G \in K[[X]][Z]$. By Lemma \ref{separablevarchangepersist} we get that $F$, hence $G$, is separable in $Z$ over $K((X))$. By Theorem \ref{abhyankarproof}, $G(T^m,Z)$ factors linearly over $K[[T]]$ for some positive integer $m$. Consequently we get that $F(T^m, Y)$ factors linearly over $K((T))$.
\end{proof}
 
 In the following we show that the elements in $K\servar{X}$ algebraic over $K(X)$ form a discrete algebraically closed field.

\begin{lem}
\label{rootorderbound}
Let $K$ be a field and $F(X,Y) = Y^n + b_1 Y^{n-1} + ... + b_n \in K(X)[Y]$ be a non-constant monic polynomial such that $b_n\neq 0$. If $\gamma \in K((T))$ is a root of $F(T^q, Y)$, then $\ord~\gamma \leq d$ for some positive integer $d$.
\end{lem}
\begin{proof}
We can find $h \in K[X]$ such that $G = h F = a_0(X) Y^n + a_1(X) Y^{n-1} + ... + a_n(X)\in K[X][Y]$ with $a_n \neq 0$. Let $d = \ord~a_n(T^q)$. If $\ord~\gamma > d$ then so is $\ord~a_i \gamma^{n-i}$ for $0 \leq i < n$. But we know that in $a_n$ there is a non-zero term with $T-$degree $d$. Thus $G(T^q, \gamma) \h 0$; Consequently $F(T^q,\gamma) \h 0$
\end{proof}

Note that if $\alpha, \beta \in K\servar{X}$ are algebraic over $K(X)$ then $\alpha + \beta$ and $\alpha \beta$ are algebraic over $K(X)$ \cite [Ch 6, Corollary 1.4] {Mines}. 

\begin{lem}
\label{discreteness}
Let $K$ be a field. The set of elements in $K\servar{X}$ algebraic over $K(X)$ is a discrete set; More precisely $\#$ is decidable on this set.
\end{lem}
\begin{proof}
It suffices to show that for an element $\gamma$ in this set $\gamma \h 0$ is decidable.
Let $F = Y^n + a_1(X) Y^{n-1} +...+a_n\in K(X)[Y]$ be a monic non-constant polynomial. Let $\gamma \in K((T))$ be a root of $F(T^q,Y)$. If $F = Y^n$ then $\neg \gamma \h 0$. Otherwise, $F$ can be written as $Y^m (Y^{n-m} + ...+a_m)$ with $0\leq m < n$ and $a_m \neq 0$.
By Lemma \ref{rootorderbound} we can find $d$ such that any element in $K((T))$ that is a root of $Y^{n-m} + ... + a_m$ has an order less than or equal to $d$. Thus $\gamma \h 0$ if an only if $\ord~\gamma \leq d$.
\end{proof}

If $\alpha\h0 \in K\servar{X}$ is algebraic over $K(X)$ then $1/\alpha$ is algebraic over $K(X)$. Thus the set of elements in $K\servar{X}$ algebraic over $K(X)$ form a field $K\servar{X}^{alg} \subset K\servar{X}$.  This field is in fact algebraically closed \emph{in} $K\servar{X}$ \cite[Ch 6, Corollary 1.5] {Mines}.

Since for an algebraically closed field $K$ we have shown $K\servar{X}$ to be only \emph{separably} algebraically closed, we need a stronger argument to show that $K\servar{X}^{alg}$ is algebraically closed.

\begin{lem}
\label{algebraicallyclosed}
For an algebraically closed field $K$ of characteristic zero, the field $K\servar{X}^{alg}$ is algebraically closed.
\end{lem}
\begin{proof}
Let $F \in K\servar{X}^{alg} [Y]$ be a monic non-constant polynomial of degree $n$. By Lemma \ref{discreteness} $K\servar{X}^{alg}$ is a discrete field. By Lemma \ref{squarefree} we can decompose $F$ as $F = H G$ with $H \in K\servar{X}^{alg} [Y]$ a non-constant monic separable polynomial. By Corollary \ref{sepalgclose}, $H$ has a root $\eta$ in $K\servar{X}$. Since $K\servar{X}^{alg}$ is algebraically closed in $K\servar{X}$ we have that $\eta \in K\servar{X}^{alg}$.
\end{proof}

We can draw similar conclusions in the case of real closed fields \footnote{We reiterate that by a field we mean a discrete field.}.
\begin{lem}
\label{seprealclose}
Let $R$ be a real closed field. Then 
\begin{enumerate}
\item For any $\alpha \h 0 \in R\servar{X}$ we can find $\beta \in R\servar{X}$ such that $\beta^2 = \alpha$ or $-\beta^2 = \alpha$. 
\item A separable monic polynomial of odd degree in $R\servar{X}[Y]$ has a root in $R\servar{X}$.
\end{enumerate}
\end{lem}
\begin{proof}
Since $R$ is real closed, the first statement follows from the fact an element $a_0 + a_1 X + ... \in R[[X]]$ with $a_0 > 0$ has a square root in $R[[X]]$.

Let $F(X,Y) = Y^n + \alpha_1 Y^{n-1} + ... + \alpha_n\in R[[X]][Y]$ be a monic polynomial of odd degree $n > 1$ separable over $R((X))$. We can assume w.l.o.g. that $\alpha_1 = 0$. Since $F$ is separable, i.e. $P F + Q F_Y = 1$ for some $P, Q \in R((X))[Y]$, then by a similar construction to that in Lemma \ref{keylemma} we can write $F(T^m,T^{p}Z)= T^{np} V$ for $V \in R[[T]][Z]$ such that $V(0,Z) \neq (Z + a)^n$ for all $a\in R$. Since $R$ is real closed and $V(0,Z)$ has odd degree, $V(0,Z)$ has a root $r$ in $R$. We can find proper decomposition into coprime factors $V(0,Z) = (Z-r)^\ell q$. By Hensel's Lemma\ref{hensel}, we lift those factors to factors of $V$ in $R[[T]][Z]$ thus we can write $F = G H$ for monic non-constant $G,H \in R[[T]][Y]$. By Lemma \ref{sepdivisorsep} both $G$ and $H$ are separable. Either $G$ or $H$ has odd degree. Assuming $G$ has odd degree greater than $1$, we can further factor $G$ into non-constant factors. The statement follows by induction.
\end{proof}

Let $R$ be a real closed field. By Lemma \ref{discreteness} we see that $R\servar{X}^{alg}$ is discrete. A non-zero element in $\alpha \in R\servar{X}^{alg}$ can be written $\alpha = X^{m/n} (a_0 + a_1 X^{1/n} + ...)$ for $n> 0, m \in \mathbb{Z}$ with $a_0 \neq 0$.  Then $\alpha$ is positive iff its initial coefficient $a_0$ is positive \cite{basu}. We can then see that this makes $R\servar{X}^{alg}$ an ordered field. 

\begin{lem}
\label{realclosed}
For a real closed field $R$, the field $R\servar{X}^{alg}$ is real closed.
\end{lem}
\begin{proof}
Let $\alpha \in R\servar{X}^{alg}$. Since $R\servar{X}^{alg}$ is discrete, by Lemma \ref{seprealclose} we can find $\beta \in R\servar{X}^{alg}$ such that $\beta^2 = \alpha$ or $-\beta^2 = \alpha$.\\
Let $F \in R\servar{X}^{alg} [Y]$ be a monic polynomial of odd degree $n$. Applying Lemma \ref{squarefree} several times, by induction we have $F = H_1 H_2..H_m$ with $H_i \in R\servar{X}^{alg} [Y]$ separable non-constant monic polynomial. For some $i$ we have $H_i$ of odd degree. By Lemma \ref{seprealclose}, $H_i$ has a root in $R\servar{X}^{alg}$. Thus $F$ has a root in $R\servar{X}^{alg}$.
\end{proof}


\section{Dynamical interpretation}
\label{section:dynamical}
 The goal of this section is to give a version of Theorem \ref{abhyankarproof} over
a field $K$ of characteristic $0$, not necessarily algebraically closed.

\begin{defn}[Regular ring]
A commutative ring $R$ is \emph{(von Neumann) regular} if for every element $a \in R$ there exist $b\in R$ 
such that  $a b a = a$ and $b a b = b$. This element $b$ is called the quasi-inverse of $a$. 
\end{defn}

 A ring is regular iff it is zero-dimensional and reduced. It is also equivalent to the fact that
any principal ideal (and hence any finitely generated ideal) is generated by an idempotent.
If $a$ is an element in $R$ and $aba=a,~bab=b$ then the element $e=ab$ is an idempotent such that
$\ideal{e}=\ideal{a}$ and $R$ is isomorphic to $R_0\times R_1$ with $R_0=R/\ideal{e}$ and
$R_1 = R/\ideal{1-e}$. Furthermore $a$ is $0$ on the component $R_0$ and invertible on the component
$R_1$.

We define strict B\'ezout rings as in \cite [Ch 4] {lombardi_book}.
\begin{defn}
A ring $R$ is a (strict) B\'ezout ring if for all $a, b \in R$ we can find $g, a_1, b_1, c, d \in R$ such that
$a = a_1 g$, $b =b_1 g$ and $c a_1+ d b_1= 1$.
\end{defn}

 If $R$ is a regular ring then $R[X]$ is a strict B\'ezout ring (and the converse is true \cite{lombardi_book}).
Intuitively we can compute the gcd as if $R$ was a field, but we may need to split $R$ when deciding
if an element is invertible or $0$. Using this, we see that given $a,b$ in $R[X]$ we can find a decomposition
$R_1,\dots ,R_n$ of $R$ and for each $i$ we have
 $g, a_1, b_1, c, d$ in $R_i[X]$ such that
$a = a_1 g$, $b =b_1 g$ and $c a_1+ d b_1= 1$ with $g$ monic. The degree of $g$ may depend on $i$.

\begin{lem}\label{extension}
If $R$ is regular and $p$ in $R[X]$ is a separable polynomial then $R[a]=R[X]/\ideal{p}$ is regular.
\end{lem}

\begin{proof}
If $c = q(a)$ is an element of $R[a]$ with $q$ in $R[X]$ we compute the gcd $g$ of $p$ and $q$.
If $p = gp_1$, we can find $u$ and $v$ in $R[X]$ such that $ug + vp_1 = 1$ since $p$ is separable.
We then have $g(a) p_1(a) = 0$ and $u(a)g(a) + v(a)p_1(a) = 1$. It follows that $e = u(a)g(a)$ is
idempotent and we have $\ideal{e} = \ideal{g(a)}$.
\end{proof}

 A {\em triangular separable $K$--algebra} 
$$R = K[a_1,\dots,a_n],~p_1(a_1) = 0,~p_2(a_1,a_2) = 0,\dots$$
is a sequence of separable extension starting from a field $K$, with $p_1$ in $K[X]$,
$p_2$ in $K[a_1][X]$, $\dots$ all monic and separable polynomials. A triangular separable
algebra is thought of as an approximation of the algebraic closure of $K$, and is determined
by a list of polynomials $p_1(X_1),~p_2(X_1,X_2),\dots$
(This is related to the way \cite{edwards2005essays} avoids the algebraic closure, by adding
only constants as needed, with the difference that we don't assume an irreducibility test.)
It follows from Lemma \ref{extension} that each triangular separable algebra
defines a regular algebra $K[a_1,\dots,a_n]$. In this case however, the idempotent elements
have a simpler direct description. If we have a decomposition
$p_l(a_1,\dots,a_{l-1},X) = g(X)q(X)$ with $g,q$ in $K[a_1,\dots,a_{l-1},X]$ then
since $p_l$ is separable, we have a relation $rg+sq = 1$ and $e = r(a_l)g(a_l),~1-e=s(a_l)q(a_l)$
are then idempotent element. We then have a decomposition of $R$ in two triangular separable algebras
$p_1,\dots,p_{l-1},g,p_{l+1},\dots$ and $p_1,\dots,p_{l-1},q,p_{l+1},\dots$. If we iterate this process
we obtain the notion of {\em decomposition} of a triangular separable algebra $R$ in finitely many
triangular algebra $R_1,\dots,R_n$. 
This decomposition stops when all polynomials $p_1,\dots,p_l$ are irreducible, i.e. when $R$ is a field.  For a triangular separable algebra $R$ and an ideal $I$ of $R$, if $R/I$ is a triangular separable algebra then we describe $R/I$ as being a \emph{refinement} of $R$. Thus a refinement of $K[a_1,...,a_n],p_1,...,p_n$ is of the form $K[b_1,...,b_n],q_1,...,q_n$ with $q_i \mid p_i$.
 
 The following is a corollary of Lemma \ref{squarefree}.
\begin{cor}
\label{squarefreemonic}
Let $f$ be a monic polynomial in $R[X]$ where $R$ is a triangular separable $K$--algebra. If $f'$ is the derivative of $f$ 
then there exist a decomposition $R_1,\dots, R_n$ and on each $R_i$
we can find polynomials $h,g,q,r,s$ in $R_i[X]$ such that $f = h g$, $f' = q g$ and $r h + s q = 1$
with $h$ monic and separable.
\end{cor}

\begin{lem}
\label{sum1decomp}
Let $R$ be a regular ring and let $a_1,...,a_n \in R$ such that $1 \in \ideal{a_1,...,a_n}$. Then we can find a decomposition  $R \cong R_1 \times ... \times R_m$ such that for each $R_i$ we have $a_j$ a unit in $R_i$ for some $1 \leq j \leq n$.
\end{lem}
\begin{proof}
We have a decomposition $R\cong A \times B$ with $a_n$ unit in $A$ and zero in $B$. We have $1 \in \ideal{a_1,...,a_{n-1}}$ in $B$. The statement follows by induction.
\end{proof}

\begin{lem}
\label{apartcontinvert}
Let $R$ be a triangular separable algebra over a field $K$ of characteristic $0$. Let $F(X,Y) = \sum_{i=0}^n \alpha_i(X) Y^{n-i}\in R[[X]][Y]$ be a monic polynomial such that $P F + Q F_Y = \gamma$ for some $P,Q \in R[[X]][Y]$ and $\gamma \h0$ in $K[[X]]$. Then we can find a decomposition $R_1,...$ of $R$ such that in each $R_i$ we have $\alpha_k(m)$ a unit for some $m$ and $k = n$ or $k = n-1$.
\end{lem}
\begin{proof}
Since $\gamma \h 0 \in K[[X]]$ we have $\gamma(\ell)$ a unit for some $\ell$. Since $P F + Q F_Y = \gamma$, we have $\eta \alpha_n + \theta \alpha_{n-1} = \gamma$ with $\eta = P(0)$ and $\theta = Q(0)$. Then we have $\sum_{i + j = \ell} \eta(i) \alpha_n(j) + \theta(i) \alpha_{n-1}(j) = \gamma(\ell)$. By Lemma \ref{sum1decomp} we have a decomposition $R_1,...$ of $R$ such that in $R_i$ we have $\alpha_k(m)$ is a unit for some $m$ and $k = n \lor k = n-1$.
\end{proof}

 Lemma \ref{keylemma} becomes in this way.

\begin{lem}\label{keylemmadyn}
Let $R$ be a triangular separable algebra over a field $K$ of characteristic $0$. Let
$F(X,Y) = Y^n + \sum_{i=1}^n \alpha_i(X) Y^{n-i}\in R[[X]][Y]$ be a
monic non-constant polynomial of degree $n \geq 2$ such that $P F + Q F_Y = \gamma$ for some $P,Q \in R[[X]][Y]$ and $\gamma \h0$ in $K[[X]]$.
There exists then a decomposition $R_1,\dots$ of $R$ and for each $i$
there exist $m>0$ and a proper factorization $F(T^m,Y)= G(T,Y)H(T,Y)$ with $G$ and $H$ in $S_i[[T]][Y]$
where $S_i = R_i[a]$ is a separable extension of $R_i$.
\end{lem}

\begin{proof}
By Lemma \ref{apartcontinvert} we have a decomposition $A_1,...$ of $R$ such that in each $A_i$ we have $\alpha_k(m)$ a unit for some $m$ and $k = n$ or $k = n-1$.
The rest of the proof proceeds as the proof of Lemma \ref{keylemma}, assuming w.l.o.g. $\alpha_1 = 0$.
We then find a decomposition of each $A_i$; thus a decomposition $R_1,\dots$ of $R$ and for each $l$
we can then find $m$ and $p$ such that $\alpha_m(p)$ is invertible and
$\alpha_i(j) = 0$ whenever $j/i  < p/m$ in $R_l$. We can then write
\begin{equation*}
F(T^m,T^{p}Z)= T^{np}(Z^n + c_2(T) Z^{n-2} + \dots + c_n(T))
\end{equation*}
with $c_m(0)$ a unit. 
We then find a further decomposition $R_{l1},R_{l2},\dots$ of $R_l$
and for each $q$ a number $s$ and a separable extension $R_{lq}[a]$ of $R_{lq}$ such that
$$Z^n + c_2(0)Z^{n-2} + \dots + c_n(0) = (Z-a)^sL(Z)$$
with $L(a)$ invertible. Using Hensel's Lemma \ref{hensel}, we can lift this to a proper decomposition 
$Z^n + c_2(T) Z^{n-2} + \dots + c_n(T) = G_1(T,Z)H_1(T,Z)$ with $G_1(T,Z)$ monic of degree $t$ and
$H_1(T,Z)$ monic of degree $u$. We take
$G(T,Y) = T^{tp}G_1(T,Y/T^p)$ and $H(T,Y) = T^{up}H_1(T,Y/T^p).$
\end{proof}
 
 We can then state the following version of Newton--Puiseux algorithm.

\begin{thm}\label{abhyankarproofdyn}
Let $K$ be a field of characteristic $0$. Let $F(X,Y) = Y^n + \sum_{i=1}^n \alpha_i(X) Y^{n-i}$ in $K[[X]][Y]$ be a
monic non-constant polynomial separable over $K((X))$. There exists then a triangular separable algebra $R$ over $K$ 
and $m>0$ and a factorization
\begin{equation*}
F(T^m,Y)= \prod_{i=1}^{n} \big(Y - \eta_i\big)\;\;\;\;\;\;\;\; \eta_i \in R[[T]]
\end{equation*}
\end{thm}

 The algorithm for computing this factorization proceeds by induction on $n$, using Lemma \ref{keylemmadyn}. 
More precisely the algorithm proceeds as follows. At a given point, we have computed
\begin{enumerate}
\item a triangular separable extension $R$ of $K$
\item a number $m$ and a partial decomposition $F(T^m,Y) = H_1(T,Y)\dots H_r(T,Y)$ with all
$H_i  \in R[[T]][Y]$ monic in $Y$.
\end{enumerate}
The algorithm stops if all $H_i$ are of degree $1$ in $Y$. Otherwise, we apply Lemma \ref{keylemmadyn} 
to the first polynomial $H_i(T,Y)$ of degree $>1$ in $Y$
to compute a decomposition 
of $R$ and for each algebra $S$ in this decomposition a separable extension $S[a]$, a positive integer $p$ and a proper decomposition
$H_i(T^p,Y) = G(T,Y)G_1(T,Y)$. We select then one algebra, and we proceed with the decomposition
$$
F(T^{mp},Y) = H_1(T^p,Y)\dots H_{i-1}(T^p,Y)G(T,Y)G_1(T,Y)H_{i+1}(T^p,Y)\dots H_r(T^p,Y)$$

\section{Analysis of the theorem}

 The previous algorithm is not deterministic when selecting an algebra in a decomposition. The goal of this section is to compare two possible
triangular separable algebras  that can be obtained by this algorithm. We are going to show
that they are both powers of a common triangular algebra.

In the following we refer to the elementary symmetric polynomials in $n$ variables by $\sigma_1,...,\sigma_n$ taking $\sigma_i(X_1,...,X_n) = \sum\limits_{1 \leq j_1 < ... j_i  \leq n} X_{j_1} ... X_{j_i}$.

\begin{lem}\label{symmrootszeros}
Let $R$ be a reduced ring. Given $a_1,...,a_n \in R$, if $\sigma_i(a_1,...,a_n) = 0$ for $0< i \leqslant n$ then $a_1=a_2=...=a_n=0$.
\end{lem}
\begin{proof}
We have $\prod\limits_{i=1}^n (X - a_i) = X^n$. Hence, $a_i^n = 0$ for $0 < i \leqslant n$ and since $R$ is reduced, $a_i = 0$.
\end{proof}

\begin{lem}\label{orderlemma}
Let $R$ be a reduced ring. Given $\alpha_1,...,\alpha_n \in R[[X]]$ such that for some positive rational 
number $d$ we have $\ord(\sigma_i(\alpha_1,...,\alpha_n)) \geqslant di$ for $0 < i \leqslant n$. 
Then $\ord(\alpha_i) \geqslant d$ for $0 < i \leqslant n$.
\end{lem}
\begin{proof}
Let $\alpha_i = \sum_{j=0}^\infty  \alpha_i(j) X^j$. We show that $\alpha_i(j) = 0$ if $j < d$.
Assume that we have $\alpha_i(j) = 0$ for $j<m< d$. We show then $\alpha_i(m) = 0$ for $i=1,\dots,n$.
The coefficient of $X^{im}$ in $\sigma_i(\alpha_1,...,\alpha_n)$ is $\sigma_i(\alpha_1(m),\dots,\alpha_n(m))$.
Since $\ord(\sigma_i(\alpha_1,...,\alpha_n)) > mi$ we get that $\sigma_i(\alpha_1(m),\dots,\alpha_n(m))=0$ and hence by Lemma 
\ref{symmrootszeros} we get that $\alpha_i(m) = 0$ for $i=1,\dots,n$.
\end{proof}

\begin{lem}
\label{splitlemma}
For a ring $R$ and a reduced extension $R\rightarrow A$, 
let $F =Y^n+\sum_{i=1}^n \alpha_i Y^{n-i}$ be an element of $R[[X]][Y]$ such that
$F(T^q,T^pZ) = T^{np} F_1(T,Z)$ with $F_1$ in $R[[T]][Z]$ for some $q>0,p$. If $F(U^m,Y)$ factors linearly over $A[[U]]$
for some $m>0$ then $F_1(0,Z)$ factors linearly over $A$.
\end{lem}
\begin{proof}
We have $F(U^m,Y) = \prod_{i=1}^{n} (Y-\eta_i)$, $\eta_i \in A[[U]]$ \\ and hence we have
$F(V^{mq},V^{mp}Z) = \prod_{i=1}^{n} (V^{mp}Z-\eta_i(V^{q}))$, $\eta_i(U) \in A[[U]]$ and 
\begin{equation*} 
F_1(V^m,Z) = \prod_{i=1}^{n} (Z-V^{-mp}\eta_i(V^q)) = Z^n + \sum_{i=1}^n V^{-imp} \beta_i(V^q) Z^{n-i}
\end{equation*} 
Since $F_1(T,Z)$ is in $R[[T]][Z]$ we have $imp\leqslant \ord~ \beta_i(V^q)$. \\
Since $\beta_i(V^q) = \sigma_i(\eta_1(V^q),\dots,\eta_n(V^q))$, 
Lemma \ref{orderlemma} shows that $mp \leqslant \ord~\eta_i(V^q)$ for $0<i\leqslant n$. 
Hence $\mu_i(V) = V^{-mp}\eta_i(V^q)$ is
in $A[[V]]$ and since $F_1(V,Z) = \prod_{i=1}^{n} (Z-\mu_i(V))$, we have that $F_1(0,Z)$ 
factors linearly over $A$, of roots $\mu_i(0)$.
\end{proof}

\begin{defn}
\label{splitdef}
Let $R=K[a_1,...,a_n],p_1,...,p_n$ be a triangular separable algebra with $p_i$ of degree $m_i$ and $A$ an algebra over $K$. Then $A$ splits $R$ if there exist a family of elements $\{a_{i_1,...,i_l} \in A \mid 0 < l \leq n, 0 < i_j \leq m_j\}$ such that \begin{align*}
& p_1=\prod_{d=0}^{m_1} (X - a_d)\\
& p_{l+1}(a_{i_1}, a_{i_1,i_2},...,a_{i_1,...,i_l},X) = \prod_{d=0}^{m_{l+1}} (X-a_{i_1,...,i_l,d})
\end{align*}
for $0 < l < n$
\end{defn}

We can view the previous definition as that of a tree of homomorphisms from the subalgebras of $R$ to $A$. At the root we have the identity homomorphism from $K$ to $A$ under which $p_1$ factors linearly, i.e. $p_1 = \prod_{j=0}^{m_1}(X-\bar{a}_{1j})$. From this we obtain $m_1$ homomorphisms $\varphi_1,...,\varphi_{m_1}$ from $K[a_1]$ to $A$ each taking $a_1$ to a different $\bar{a}_{1j}$. If $p_2$ factors linearly under say $\varphi_1$, i.e. $\varphi_1(p_2) = \prod_{j=0}^{m_2} (X-\bar{a}_{2j})$ then we obtain $m_2$ different (since $p_2$ is separable) homomorphisms $\varphi_{11}, ..., \varphi_{1m_2}$ from $K[a_1,a_2]$ to $A$. Similarly we obtain $m_2$ different homomorphisms from $K[a_1,a_2]$ to $A$ by extending $\varphi_2,\varphi_3,...etc$, thus having $m_1m_2$ homomorphism in total. Continuing in this fashion we obtain the $m$ different homomorphisms of the family $\mathcal{S}$.

We note that if an algebra $A$ over $K$ splits a triangular separable algebra $R$ over $K$ then $A\otimes_K R \cong A^{[R:K]}$. If $A$ is a field then the converse is also true as the following lemma shows.

\begin{lem}\label{splitfield}
Let $L/K$ be a field and $R=K[a_1,...,a_n],p_1,...,p_n$ a triangular separable algebra. Then $L \otimes_K R \cong L^{[R:K]}$ only if $L$ splits $R$.
\end{lem}
\begin{proof}
Let $\deg(p_i) = m_i$, $[R:K] = m = \prod_{i=1}^n m_i$ and let $L \otimes_K R\cong L^{[R:K]}$. Then there exist a system of orthogonal idempotents\footnote{That is $e_i e_j = 0$ if $i \neq j$ and $e_1 + ... + e_m = 1$.} $e_1,...,e_m$ such that $A=L \otimes_K R\cong A/(1-e_1) \times .... \times A/(1-e_m) = L^m$. Let $a_{ij}$ be the image of $a_i$ in $A/(1-e_j)$. Then we have $(a_{11},...,a_{n1}) \neq (a_{12},...,a_{n2}) \neq ... \neq (a_{1m},...,a_{nm})$ since otherwise we will have the ideals $\ideal{1-e_i}=\ideal{1-e_j}$ for some $i \neq j$. Since $p_1$ is separable there are up to $m_1$ different images $a_{1j}$ of $a_1$. Thus the size of the set $\{a_{1j} \mid 0 < j \leq m\}$ is equal to $m_1$ only if $p_1$ factors linearly over $L$. Similarly, for each different image $\bar{a}_1$ of $a_1$ there are up to $m_2$ possible images of $a_2$ in $L$ since the polynomial $p_2(\bar{a}_1,X)$ is separable. Thus the size of the set $\{(a_{1j},a_{2j}) \mid 0 < j \leq m\}$ is equal $m_1m_2$ only if $p_1$ factors linearly over $L$ and for each root $\bar{a}_1$ of $p_1$ the polynomial $p_2(\bar{a}_1,X)$ factors linearly over $L$. Continuing in this fashion we find that the size of the set $\{(a_{1j},...,a_{nj}) \mid 0 < j \leq m\}$ is equal to $m_1...m_n = m$ only if $L$ splits $R$.
\end{proof}

\begin{lem}\label{decompsplitdiv}
Let $A$ be a triangular separable algebra over a field $K$ and let $p$ be a monic non-constant polynomial of degree $m$ in $A[X]$ such that $p = \prod_{i=1}^m (X-a_i)$ with $a_i \in A$. If $g$ is a monic non-constant polynomial of degree $n$ such that $g \mid p$ then we have a decomposition $A \cong R_1 \times ... \times R_l$ such that for any $R_j$ in the product $g = \prod_{i=1}^{n} (X-\bar{a}_i)$ with $\bar{a}_i \in R_j$ the image in $R_j$ of some $a_k, 0 < k \leq m$.
\end{lem}
\begin{proof}
Let $p = (X-a_1)...(X-a_n)$ for $a_1,...,a_n \in A$. Let $p = g q$. Then $p(a_1) = g(a_1) q(a_1) = 0$. We can find a decomposition of $A$ into triangular separable algebras $A_1 \times ... \times A_t \times B_1 \times B_s$ such that $g(a_1) = 0$ in $A_i, 0 < i \leq t$ and $g(a_1)$ is a unit in $B_i, 0 < i \leq s$ in which case $q(a_1) = 0$ in $B_i$. By induction we can find a decomposition of $A$ into a product of triangular separable algebras $R_1,\dots,R_l$ such that $g$ factors linearly over $R_i$. 
\end{proof}

From Definition \ref{splitdef} it is obvious that if an algebra $A$ splits a triangular separable algebra $R$ then $A/I$ splits $R$ for any ideal $I$ of $A$. 

\begin{lem}\label{splitpersist}
Let $A$ and $R$ be triangular separable algebras over $K$ such that $A$ splits $R$. Let $B$ be a refinement of $R$. Then we can find a decomposition $A \cong A_1 \times ... \times A_m$ into a product of triangular separable algebras such that $A_i$ splits $B$ for $0<i\leq m$.
\end{lem}
\begin{proof}
Let $R = K[a_1,...,a_n],p_1,...,p_n$. Then $B = K[\bar{a}_1,...,\bar{a}_n], g_1,...,g_n$ where $g_j \mid p_j$ for $0 < j \leq n$. Let $\deg(p_j) = m_j$ and $\deg(g_j) = \ell_j$ for $0< j\leq n$. Since $A$ splits $R$ we have a family of elements $\{a_{i_1,...,i_l} \in A \mid 0 < l \leq n, 0 < i_j \leq m_j\}$ satisfying the condition of Definition \ref{splitdef}. we have $p_1=\prod_{i=1}^{m_1} (X-a_{i_1})$. By Lemma \ref{decompsplitdiv} we decompose $A$ into the product $A_1 \times...\times A_t$ such that for any given $A_k$ in the product we have $p = \prod_{i=1}^{m_1} (X-\bar{a}_{i_1})$ and $g = \prod_{i=1}^{\ell_1} (X-\bar{a}_{i_1})$ with $\bar{a}_{i_1} \in A_k$ for $0 < i \leq m_1$. Since each $\bar{a}_{i_1}$ is an image of some $a_{j_1}$ and $p_2(a_{j_1},X)$ factors linearly over $A$ we have that $p_2(\bar{a}_{i_1},X)$ factors linearly over $A_k$ but then $g_2(\bar{a}_{i_1},X)$ divides $p_2(\bar{a}_{i_1},X)$ and thus by Lemma \ref{decompsplitdiv} we can decompose $A_k$ into the product $B_1 \times ... \times B_s$ such that for a given $B_r$ in the product we have $p_2(\bar{a}_{i_1},X) = \prod_{j=1}^{m_2} (X - \bar{a}_{i_1,j_2})$ and $g_2(\bar{a}_{i_1},X) =  \prod_{j=1}^{\ell_2} (X - \bar{a}_{i_1,j_2})$. By induction on the $m_1$ values of $\bar{a}_{i_1}$ we can find a decomposition $D_1 \times ... \times D_l$ such that in each $D_i$ we have $g_1(X) = \prod_{i=1}^{\ell_1}(X-\bar{a}_{i_1})$ and  $g_2(\bar{a}_{i_1},X) = \prod_{j=1}^{\ell_2} (X-\bar{a}_{i_1,j_2})$ for $0 < i \leq \ell_1$. Continuing in this fashion we can find a decomposition of $A$ such that each algebra in the decomposition splits $B$.
\end{proof}

\begin{lem}\label{decompsplitsplit} Let $A$ and $B$ be triangular separable algebras such that $A\cong A_1 \times ... \times A_t$ and each $A_i$ splits $B$. Then $A$ splits $B$.
\end{lem}
\begin{proof}
Let $B = K[a_1,...,a_n],g_1,...,g_n$ with $\deg(g_i) = m_i$. Then we have a family of elements $\{a^{(i)}_{k_1,...,k_l} \mid 0 < k_j \leq m_j, 0 < j \leq n \}$ in $A_i$ satisfying the conditions of Definition \ref{splitdef}. We claim that the family \begin{equation*}\mathcal{S}=\{a_{k_1,...,k_l} \mid a_{k_1,...,k_l} = (a^{(1)}_{k_1,...,k_l},...,a^{(t)}_{k_1,...,k_l}), 0 < k_j \leq m_j, 0 < j \leq n\}\end{equation*} of $A$ elements satisfy the conditions of Definition \ref{splitdef}. Since we have a factorization $g_1 = \prod_{l=1}^{m_1} (X-a^{(i)}_l)$ over $A_i$, we have a factorization $g_1 = \prod_{l=1}^{m_1} (X-(a^{(1)}_l,...,a^{(t)}_l) =\prod_{l=1}^{m_1} (X-a_l)$ over $A$. Since for $0 < l \leq m_1$ we have a factorization $g_2(a^{(i)}_l,X) = \prod_{j=1}^{m_2} (X-a^{(i)}_{l,j})$ of in $A_i$, we have a factorization $g_2(a_l,X) = \prod_{j=1}^{m_2} (X-(a^{(1)}_{l,j},...,a^{(t)}_{l,j}) = \prod_{j=1}^{m_2} (X-a_{l,j})$. Continuing in this fashion we verify that the family $\mathcal{S}$ satisfy the requirements of Definition \ref{splitdef}.
\end{proof}
\begin{cor}\label{splitalgsplitrefinement}
Let $A$ and $B$ be triangular separable algebras such that $A$ splits $B$. Then $A$ splits any refinement of $B$.
\end{cor}
Lemmas \ref{splitlemma}, \ref{decompsplitsplit} and Corollary \ref{splitalgsplitrefinement} allow us to extend Lemma \ref{keylemmadyn} as follows.

\begin{lem}\label{keylemmadynext}
Let $R=K[a_1,...,a_n],p_1,...,p_n$ be a triangular separable algebra with $\deg(p_i) = m_i$. Let
$F(a_1,...,a_n, X,Y) = Y^n + \sum_{i=1}^n \alpha_i(X) Y^{n-i}\in R[[X]][Y]$ be a
monic non-constant polynomial of degree $n \geq 2$ such that $P F + Q F_Y = \gamma$ for some $P,Q \in R[[X]][Y]$, $\gamma \in R[[X]]$ with $\gamma \h 0$. 
There exists then a decomposition $R_1,\dots$ of $R$ and for each $i$
there exist $m>0$ and a proper factorization $F(T^m,Y)= G(T,Y)H(T,Y)$ with $G$ and $H$ in $S_i[[T]][Y]$
where $S_i = R_i[b],q$ is a separable extension of $R_i$. \\ Moreover, Let $A$ be a triangular separable algebra such that $A$ splits $R$ and let $\{a_{i_1,...,i_l}\mid 0 < l \leq n, 0 < i \leq m_i\}$ be the family of elements in $A$ satisfying the conditions in Definition $\ref{splitdef}$. If $F(a_{i_1},...,a_{i_1,...,i_n}, X,Y)$ factors linearly over $A[[U]]$ for $0 < i \leq m_i$ where $U^v = X$ for some positive integer $v$ then $A$ splits $S_i$.
\end{lem}
\begin{proof}
The proof proceeds as the proof of Lemma \ref{keylemma}, assuming w.l.o.g. $\alpha_1 = 0$.
We first find a decomposition $R_1,\dots$ of $R$ and for each $l$
we can then find $m$ and $p$ such that $\alpha_m(p)$ is invertible and
$\alpha_i(j) = 0$ whenever $j/i  < p/m$ in $R_l$. We can then write
\begin{equation*}
F(T^m,T^{p}Z)= T^{np}(Z^n + c_2(T) Z^{n-2} + \dots + c_n(T))
\end{equation*}
with $\ord~c_m = 0$.
Since $A$ splits $R$ then by Lemma\ref{splitpersist} we can find a decomposition $A_1,\dots$ of $A$ such that each $A_i$ splits $R_l$ for each $l$.
We then find a further decomposition $R_{l1},R_{l2},\dots$ of $R_l$
and for each $t$ a number $s$ and a separable extension $R_{lt}[a]$ of $R_{lt}$ such that
$$q=Z^n + c_2(0)Z^{n-2} + \dots + c_n(0) = (Z-a)^sL(Z)$$
with $L(a)$ invertible. Similarly, we can decompose each $A_i$ further into $B_1, \dots$ such that each $B_i$ splits each $R_{lt}$ for all $l,t$. Let the family $\mathcal{F}=\{b_{i_1,...,i_l}\mid 0 < l \leq m, 0 < i \leq m_i\}$ be the image of the family $\{a_{i_1,...,i_l}\mid 0 < l \leq n, 0 < i \leq m_i\}$ in $B_i$. Then $B_i$ splits $R$ with $\mathcal{F}$ as the family of elements of $B_i$ satisfying Definition \ref{splitdef}. But then $F(b_{i_1},...,b_{i_1,...,i_n}, X,Y)$ factors linearly over $B_i$. For some subfamily $\{c_{i_1},...,c_{i_1,...,i_l} \mid 0 < l \leq n, 0 < i_j \leq \bar{m}_j \leq m_j\}\subset \mathcal{F}$ of elements in $B_i$ we have that $B_i$ splits $R_{lt}$. Thus $F(c_{i_1},...,c_{i_1,...,i_n}, X,Y)$ factors linearly over $B_i$ for all $c_{i_1},...,c_{i_1,...,i_n}$ in the family. By Lemma \ref{splitlemma} we have that $q(c_{i_1},...,c_{i_1,...,i_n},Z)$ factors linearly over $B_i$ for all $c_{i_1},...,c_{i_1,...,i_n}$. Thus $B_i$ splits the extension $R_{lt}[a]$. But then by Lemma \ref{decompsplitsplit} we have that $A$ splits $R_{lt}[a]$.
Using Hensel's Lemma \ref{hensel}, we can lift this to a proper decomposition 
$Z^n + c_2(T) Z^{n-2} + \dots + c_n(T) = G_1(T,Z)H_1(T,Z)$ with $G_1(T,Z)$ monic of degree $t$ and
$H_1(T,Z)$ monic of degree $u$. We take
$G(T,Y) = T^{tp}G_1(T,Y/T^p)$ and $H(T,Y) = T^{up}H_1(T,Y/T^p).$
\end{proof}
 
 We can then extend Theorem \ref{abhyankarproofdyn} as follows.
 
\begin{thm}\label{dynamicnewton-extended}
Let $F(X,Y) = Y^n + \sum_{i=1}^n \alpha_i(X) Y^{n-i}\in K[[X]][Y]$ be a
monic non-constant polynomial separable over $K((X))$. There exists then a triangular separable algebra $R$ over $K$ 
and $m>0$ and a factorization
\begin{equation*}
F(T^m,Y)= \prod_{i=1}^{n} \big(Y - \eta_i\big)\;\;\;\;\;\;\;\; \eta_i \in R[[T]]
\end{equation*}
Moreover, if $A$ is a triangular separable algebra over $K$ such that $F(X,Y)$ factors linearly over $A[[X^{1/s}]]$ for some positive integer $s$ then $A$ splits $R$.
\end{thm} 
 
As we shall see in the examples below, the result of the computation is usually several triangular separable algebras $R_1,...$ over the base field $K$ with linear factorizations of $F$ over $R_i[[X^{1/r}]],...$ for some $r \in \mathbb{Z}^+$. The previous theorem allows us to state the following about these algebras. 

\begin{cor}\label{spliteachother}
Let $A$ and $B$ be two triangular separable algebras obtained by the algorithm of Theorem \ref{abhyankarproofdyn}. Then $A$ splits $B$ and $B$ splits $A$. Consequently, a triangular separable algebra obtained by this algorithm splits itself.
\end{cor}

Thus given any two algebras $R_1$and $R_2$ obtained by the algorithm and two prime ideals $P_1 \in \spec (R_1)$ and $P_2 \in \spec(R_2)$ we have a field isomorphism $R_1/P_1 \cong R_2/P_2$. Therefore all the algebras obtained are approximations of the same field $L$. Since $L$ splits all the algebras and itself is a refinement, $L$ splits itself, i.e. $L \otimes_K L \cong L^{[L:K]}$ and $L$ is a normal, in fact a Galois extension of $K$. 

Classically, this field $L$ is the field of constants generated over $K$ by the set of coefficients of the Puiseux expansions of $F$. The set of Puiseux expansions of $F$ is closed under the action of $\text{Gal}(\bar{K}/K)$, where $\bar{K}$ is the algebraic closure of $K$. Thus the field of constants generated by the coefficients of the expansions of $F$ is a Galois extension. The algebras generated by our algorithm are powers of this field of constants, hence are in some sense minimal extensions.

Even without the notion of prime ideals we can still show interesting relationship between the algebras produced by the algorithm of Theorem \ref{abhyankarproofdyn}. The plan is to show that any two such algebras $A$ and $B$ are essentially isomorphic in the sense that each of them is equal to the power of some common triangular separable algebra $R$, i.e. $A \cong R^m$ and $B \cong R^n$ for some positive integers $m,n$. To show that $A \cong R^m$ we have to be able to decompose $A$. To do this we need to constructively obtain a system of orthogonal nontrivial (unless $A \cong R$ already) idempotents $e_1,...,e_m$. Since $A$ and $B$ split each other, the composition of these maps gives a homomorphism from $A$ to itself. We know that a homomorphism between a field and itself is an automorphism thus as we would expect if there is a homomorphism from a triangular separable algebra $A$ to itself that is not an automorphism we can decompose this algebra non trivially. We use the composition of the split maps from $A$ to $B$ and vice versa as our homomorphism this will enable us to repeat the process after the initial decomposition, that is if $A/e_1, B/e_2$ are algebras in the decompositions of $A$ and $B$, respectively, we know that they split each other. This process of decomposition stops once we reach the common algebra $R$.

\begin{lem}
\label{dynamicnormal1}
Let $A$ be a triangular separable algebra over a field $K$ and let $\pi:A\rightarrow A$ be $K$--homomorphism. Then $\pi$ is either an automorphism of $A$ or we can find a non-trivial decomposition $A \cong A_1 \times ... \times A_t$.
\end{lem}
\begin{proof}
Let $A = K[a_1,...,a_l],p_1,...,p_l$ with $\deg(p_i) = n_i, 0<i\leq l$. Let $\pi$ map $a_i$ to $\bar{a}_i$, for $0 < i \leq l$. Then $\bar{a}_i$ is a root of $\pi(p_i) =p_i(\bar{a}_1,...,\bar{a}_{i-1},X)$. The set of vectors $\mathcal{S}=\{a_1^{i_1}...a_l^{i_l} \mid 0 \leq i_j < n_j, 0< j \leq l\}$ is a basis for the vector space $A$ over $K$. If the image $\pi(\mathcal{S}) =\{\bar{a}_1^{i_1}...\bar{a}_l^{i_l} \mid 0 \leq i_j < n_j, 0< j \leq l\}$ is a basis for $A$, i.e. $\pi(\mathcal{S})$ is a linearly independent set then $\pi$ is surjective and thus an automorphism. \\
Assuming $\pi$ is not an automorphism, then the kernel of $\pi$ is non-trivial, i.e. we have a non-zero non-unit element in $\ker \pi$, thus we have a non-trivial decomposition of $A$. 
\end{proof}

\begin{thm}
\label{dynamicnormal2}
Let $A,B$ be triangular separable algebras over a field $K$ such that $A$ splits $B$ and $B$ splits $A$. Then there exist a triangular separable algebra $R$ over $K$ and two positive integers $m,n$ such that $A\cong R^n$ and $B\cong R^m$.
\end{thm}
\begin{proof}
First we note that by Corollary \ref{splitalgsplitrefinement} if $A$ splits $B$ then $A$ splits any refinement of $B$. Trivially if $A$ splits $B$ then any refinement of $A$ splits $B$. 
Since $A$ and $B$ split each other then there is $K$--homomorphisms $\vartheta : B \rightarrow A$ and $\varphi : A \rightarrow B$. The maps $\pi = \vartheta \circ \varphi$ and $\varepsilon = \varphi \circ \vartheta$ are $K$--homomorphisms from $A$ to $A$ and $B$ to $B$ respectively. If both $\pi$ and $\varepsilon$ are automorphisms then we are done. Otherwise, by Lemma \ref{dynamicnormal1} we can find a decomposition of either $A$ or $B$. By induction on $\dim(A) + \dim(B)$ the statement follows.
\end{proof}
Theorems \ref{dynamicnormal2} and \ref{dynamicnewton-extended} show that the algebras obtained by the algorithm of Theorem \ref{abhyankarproofdyn} are equal to the power of some common algebra. This common triangular separable algebra is an approximation, for lack of irreducibility test for polynomials, of the normal field extension of $K$ generated by the coefficients of the Puiseux expansions $\eta_i \in \bar{K}[[X^{1/m}]]$ of $F$, where $\bar{K}$ is the algebraic closure of $K$.

The following are examples from a Haskell implementation of the algorithm. We truncate the different factors unevenly for readability.

\begin{exmp}
Applying the algorithm to $F(X,Y) = Y^4 - 3 Y^2 + X Y + X^2 \in Q[X][Y]$ we get.
\begin{align*}
&\bullet\;Q[a,b,c],a=0,\; b^2-13/36=0,\; c^2-3=0  \\
& F(X,Y) =\\
& \;\;\; (Y+(-b-1/6)X+(-31b/351-7/162)X^3+(-415b/41067-29/1458)X^5+...) \\
& \;\;\;   (Y+(b-1/6)X+(31b/351-7/162)X^3+(1415b/41067-29/1458)X^5+...) \\
&\;\;\;   (Y-c+X/6+5cX^2/72+7X^3/162+185cX^4/10368+29X^5/1458+...)\\
&\;\;\;   (Y+c+X/6-5cX^2/72+7X^3/162-185cX^4/10368+29X^5/1458+...)
\end{align*}
\begin{align*}
&\bullet\;Q[a,b,c], a^2-3=0,b-a/3=0,c^2-13/36=0\\
&F(X,Y)=\\
&\;\;\; (Y-a+X/6+5aX^2/72+7X^3/162+185aX^4/10368+29X^5/1458+...) \\
&\;\;\;  (Y+(-c-1/6)X+(-31c/351-7/162)X^3+(-415c/41067-29/1458)X^5+...)\\
&\;\;\;  (Y+(c-1/6)X+(31c/351-7/162)X^3+(1415c/41067-29/1458)X^5+...)\\
&\;\;\;  (Y+a+X/6-5aX^2/72+7X^3/162-185aX^4/10368+29X^5/1458+...) 
\end{align*}
\begin{align*}
&\bullet\;Q[a,b,c], a^2-3=0,b+2a/3=0,c^2-13/36=0\\
&F(X,Y) = \\
&\;\;\;(Y-a+X/6+5aX^2/72+7X^3/162+185aX^4/10368+29X^5/1458+...)\\
&\;\;\; (Y+a+X/6-5aX^2/72+7X^3/162-185aX^4/10368+29X^5/1458+...)\\
&\;\;\; (Y+(-c-1/6)X+(-31c/351-7/162)X^3+(-415c/41067-29/1458)X^5+...) \\
&\;\;\; (Y+(c-1/6)X+(31c/351-7/162)X^3+(1415c/41067-29/1458)X^5+...)
\end{align*}
\end{exmp}
The algebras in the above example can be readily seen to be isomorphic. However, as we will show next, this is not always the case.
\begin{exmp}
To illustrate Theorem \ref{dynamicnormal2} we show how it works in the context of an example computation. The polynomial is $ F(X,Y) = Y^6 + X^6 + 3 X^2 Y^4 + 3 X^4 Y^2 - 4 X^2 Y^2$. The following are two of the several triangular separable algebras obtained by our algorithm along with their respective factorization of $F(X,Y)$.

\begin{align*}
& A=Q[a,b,c,d,e], p_1,p_2,p_3,p_5\\
& p_1=Y^4-4, \; p_2 = Y-a/5, \; p_3 = Y^2-1/4,\\
& p_4 = Y^3+2 a^2Y/3+20 a^3/27,\; p_5 = Y^2+3 d^2/4+2 a^2/3 \\
& F(X,Y) = \big (Y-aX^{\tfrac{1}{2}}+3a^3X^{\tfrac{3}{2}}/16 + ...\big) \big(Y-cX^2+ ...\big) \big(Y+cX^2+ ...\big) \\ 
& \;\;\;\;  \big (Y+(-d+a/3)X^{\tfrac{1}{2}}+(-3ad^2/16-a^2d/16-7a^3/48)X^{\tfrac{3}{2}}+ ...\big) \\
& \;\;\;\; \big (Y+(-e+d/2+a/3)X^{\tfrac{1}{2}}+\\
& \;\;\;\;\;\;\;\;\;\;\;\;(3ade/16-a^2e/16+3ad^2/32+a^2d/32-a^3/48)X^{\tfrac{3}{2}}+ ...\big) \\
& \;\;\;\;  \big (Y+(e+d/2+a/3)X^{\tfrac{1}{2}}+\\
& \;\;\;\;\;\;\;\;\;\;\;\;(-3ade/16+a^2e/16+3ad^2/32+a^2d/32-a^3/48)X^{\tfrac{3}{2}}+ ...\big) \\
\\
& B=Q[r,t,u,v,w], q_1,q_2,q_3,q_5\\
& q_1= Y^4-4,\; q_2= Y+4r/5,\; q_3= Y,\;q_4=Y^2-1/4,\; q_5=Y^2+r^2 \\
& F(X,Y) = (Y-rX^{\tfrac{1}{2}}+3r^3X^{\tfrac{3}{2}}/16+ ...) (Y+rX^{\tfrac{1}{2}}-3r^3X^{\tfrac{3}{2}}/16+ ...) \\
&\;\;\;\; (Y-vX^2+ ...) (Y+vX^2+ ...) \\
&\;\;\;\; (Y-wX^{\tfrac{1}{2}}-3r^2wX^{\tfrac{3}{2}}/16+ ...) (Y+wX^{\tfrac{1}{2}}+3r^2wX^{\tfrac{3}{2}}/16+ ...)
\end{align*}

We now show that the two algebras indeed split each other. Over $B$ the polynomial $p_1$ factors as $p_1 = (Y-r) (Y+r) (Y-w) (Y+w)$. Each of these factors partly specify a homomorphism taking $a$ to a zero of $p_1$ in $B$. For each we get a factorization of $p_4$ over $B$.

\begin{itemize}
\item $a \mapsto r$\\
$p_4 = (Y + 2r/3) (Y - w - r/3) (Y + w - r/3)$
\item $a \mapsto - r$\\
$p_4 = (Y - 2r/3) (Y - w + r/3) (Y + w + r/3)$
\item $a \mapsto w$\\
$p_4 = (Y - r - w/3) (Y+ r - w/3) (Y +2w/3)$
\item $a \mapsto -w$\\
$p_4 = (Y - r + w/3) (Y + r + w/3) (Y - 2w/3)$
\end{itemize}
For each of the $4$ mappings of $a$ we get $3$ mappings of $d$. Now we see we have $12$ different mappings arising from the different mappings of $a$ and $d$. Each of these $12$ mappings will give rise to $2$ different mappings of $e$ (factorization of $p_5$)...etc. Thus we have a number of homomorphisms equal to the dimension of the algebra, that is $48$ homomorphisms. We avoid listing all these homomorphisms here. In conclusion, we see that $B$ splits $A$. Similarly, we have that $A$ splits $B$. We show only one of the $16$ homomorphisms below.
The polynomial $q_1$ factors linearly over $A$ as $q_1 = (Y - a) (Y-d + a/3) (Y - e + d/2 + a/3) (Y + e + d/2 + a/3)$. Under the map $r \mapsto a$ we get a factorization of $q_5$ over $A$ as 
\begin{align*}
q_5 = Y^2 + a^2 =& (Y-a^2d^2e/8+a^3de/12-5e/9-a^3d^2/8-2d/3-2a/9) \\
& (Y + a^2d^2e/8-a^3de/12+5e/9+a^3d^2/8+2d/3+2a/9)
\end{align*}

Now to the application of Theorem \ref{dynamicnormal2}. Under the map above we have an endomorphism $a\mapsto r \mapsto a$ and $d \mapsto -2 r/3 \mapsto -2a/3$. Thus in the kernel we have the non-zero element $d+2a/3$ and as expected $Y+2a/3$ divides $p_4$. Using this we obtain a decomposition of $A \cong A_1 \times A_2$. We have $A_1 = Q[a,b,c,d,e],p_1, p_2,p_3,g_4,p_5$ with $g_4 = Y+2a/3$ and $A_2 = Q[a,b,c,d,e],p_1, p_2,p_3,h_4,p_5$ with $h_4 = Y^2-2aY/3+10a^2/9$.

With $d+2a/3 = 0$ in $A_1$, $p_5 =  Y^2+3 d^2/4+2 a^2/3 = Y^2 + a^2$ and we can see immediately that $A_1 \cong B$. Similarly, we can decompose the algebra $A_2 \cong C_1 \times C_2$, where $C_1 = Q[a,b,c,d,e],p_1,p_2,p_3,h_4, g_5$ with $g_5 = Y-d/2+2a/3$ and $C_2 = Q[a,b,c,d,e],p_1,p_2,p_3,h_4, h_5$ with $h_5 = Y + d/2 - 2a/3$. The polynomial $q_5$ factors linearly over both $C_1$ and $C_2$ as $q_5 = (Y -d + a/3) (Y +d - a/3)$. We can readily see that both $C_1$ and $C_2$ are isomorphic to $B$, through the $C_1$ automorphism $a\mapsto r \mapsto a, d\mapsto w + r/3 \mapsto d$. Thus proving $A \cong B^3$. 
\end{exmp}

\section{Acknowledgments}
We are grateful to Henri Lombardi for the useful comments and discussions during the work leading to this paper. We thank the referees for comments and suggestions that helped to improve the article.

The research leading to the results presented here has been supported by ERC Advanced grant project 247219.

\bibliographystyle{jloganal}

\end{document}